\documentclass[a4paper,12pt]{article}
\usepackage[english]{babel}
\usepackage{amssymb,amsmath,amsthm}
\usepackage{enumitem}
\usepackage{booktabs,xcolor}
\usepackage{tikz,graphicx,color}
\usepackage{subcaption}
\usepackage{bm}
\usepackage{xurl}
\usepackage{authblk}
\usepackage{hyperref}
\usepackage{aliascnt}
\usepackage[nameinlink]{cleveref}
\usepackage[hmargin=2.3cm,vmargin=2.36cm]{geometry}
\usepackage[square,sort,comma,numbers]{natbib}

\allowdisplaybreaks

\hypersetup{
    colorlinks=true,
    linkcolor=black,
    citecolor=black,
    urlcolor=black
}

\theoremstyle{plain}
\newtheorem{theorem}{Theorem}[section]
\newaliascnt{lemma}{theorem}
\newtheorem{lemma}[lemma]{Lemma}
\aliascntresetthe{lemma}

\newaliascnt{corollary}{theorem}
\newtheorem{corollary}[corollary]{Corollary}
\aliascntresetthe{corollary}
\newaliascnt{proposition}{theorem}

\aliascntresetthe{proposition}
\newaliascnt{claim}{theorem}
\newtheorem{claim}[claim]{Claim}
\aliascntresetthe{claim}

\newaliascnt{conjecture}{theorem}
\newtheorem{conjecture}[conjecture]{Conjecture}
\aliascntresetthe{conjecture}
\newaliascnt{observation}{theorem}
\newtheorem{observation}[observation]{Observation}
\aliascntresetthe{observation}

\theoremstyle{definition}
\newaliascnt{definition}{theorem}

\aliascntresetthe{definition}

\theoremstyle{remark}
\newaliascnt{remark}{theorem}

\aliascntresetthe{remark}

\numberwithin{equation}{section}
\crefname{theorem}{Theorem}{Theorems}
\Crefname{theorem}{Theorem}{Theorems}
\crefname{lemma}{Lemma}{Lemmas}
\Crefname{lemma}{Lemma}{Lemmas}
\crefname{corollary}{Corollary}{Corollaries}
\Crefname{corollary}{Corollary}{Corollaries}
\crefname{proposition}{Proposition}{Propositions}
\Crefname{proposition}{Proposition}{Propositions}
\crefname{claim}{Claim}{Claims}
\Crefname{claim}{Claim}{Claims}
\crefname{conjecture}{Conjecture}{Conjectures}
\Crefname{conjecture}{Conjecture}{Conjectures}
\crefname{definition}{Definition}{Definitions}
\Crefname{definition}{Definition}{Definitions}
\crefname{remark}{Remark}{Remarks}
\Crefname{remark}{Remark}{Remarks}
\Crefname{observation}{Observation}{Observations}

\newcommand{\abs}[1]{\lvert#1\rvert}

\newlength{\bibitemsep}
\newlength{\bibparskip}
\let\oldthebibliography\thebibliography
\renewcommand\thebibliography[1]{%
  \oldthebibliography{#1}%
  \setlength{\parskip}{\bibitemsep}%
  \setlength{\itemsep}{\bibparskip}%
}

\newenvironment{proof*}[1][Proof]
{\begin{proof}[#1]}
{\end{proof}}

\definecolor{myblue}{rgb}{0.0627, 0.7686, 0.9882}

\begin{document}

\title{Frustration index of a signed planar graph \\ and the feedback vertex set}

\author{Sirui Chen} 

\author{Jiaao Li}

\author{Zhouningxin Wang}

\affil{\small School of Mathematical Sciences and LPMC, Nankai University, Tianjin 300071, China. \linebreak Emails: sirui.chen@mail.nankai.edu.cn, \{lijiaao, wangzhou\}@nankai.edu.cn.}

\date{}

\maketitle

\begin{abstract}
A feedback vertex set of a graph is a set of vertices whose deletion leaves a forest. In 2016, Dross, Montassier, and Pinlou conjectured that every planar graph $G$ of girth at least $g$ admits a feedback vertex set of size at most $e(G)/g$. In this note, we confirm this conjecture by connecting this problem with signed graphs. The frustration index of a signed graph $(G,\Sigma)$ is defined as the minimum number of negative edges among all signatures on $G$ that are switching-equivalent to $\Sigma$. Equivalently, it is the minimum number of edges whose deletion results in a balanced subgraph of $(G,\Sigma)$. We show that the minimum size of a feedback vertex set of a planar graph is bounded above by the maximum frustration index over all signatures of the graph, and thereby provide a tight upper bound on the size of the minimum feedback vertex set, which resolves the conjecture of Dross, Montassier, and Pinlou (2016).
\end{abstract}


\section{Introduction}

A \emph{feedback vertex set} of a graph $G$ is a set $S\subseteq V(G)$ such that $G-S$ is a forest. Let $$\mathrm{fvs}(G) = \min \{|S|:S\subseteq V(G),~S~\text{is a feedback vertex set of}~G\}.$$
For a given graph $G$ and an integer $k$, deciding whether $\mathrm{fvs}(G)\le k$ is NP-complete \cite{LY1980}. For any graph $G$, we use $v(G)$ and $e(G)$ to denote its number of vertices and edges, respectively. Various results concern structural upper bounds on $\mathrm{fvs}(G)$ for different graph classes. For planar graphs, a well-known conjecture of Albertson and Berman~\cite{AB1979} states that every planar graph $G$ has a feedback vertex set of size at most $\frac{v(G)}{2}$.

\begin{conjecture}[Albertson and Berman~\cite{AB1979}]\label{conj:planar}
Every planar graph $G$ satisfies $\mathrm{fvs}(G)\leq \frac{v(G)}{2}$.
\end{conjecture}

This conjecture, if true, would be tight as witnessed by $K_4$. Despite considerable effort, the best known general upper bound is $\mathrm{fvs}(G) \leq \frac{3v(G)}{5}$, a consequence of Borodin's result~\cite{B1979} that every planar graph admits an acyclic $5$-coloring. For triangle-free planar graphs, further improvements were obtained by Salavatipour \cite{S2006}, proving that $\mathrm{fvs}(G)\le \frac{15v(G)-24}{32}$, and Dross, Montassier, and Pinlou \cite{DMP2019} showed that $\mathrm{fvs}(G)\le \frac{5v(G)-7}{11}$, and Le~\cite{L2018} improved it to $\mathrm{fvs}(G)\le \frac{4v(G)}{9}$. In the setting of triangle‑free planar graphs and even more so for bipartite planar graphs, the following conjecture naturally arises independently by Akiyama and Watanabe\cite{AW1987}, and Albertson and Hass.

\begin{conjecture}[Akiyama and Watanabe~\cite{AW1987}, and Albertson and Hass~\cite{AH1998}]\label{conj:bipartite}
Every bipartite planar graph $G$ satisfies that $\mathrm{fvs}(G)\leq \frac{3v(G)}{8}$.
\end{conjecture}

For bipartite graphs, Alon \cite{A2003} gave an upper bound $\mathrm{fvs}(G)\leq \left(\frac{1}{2}-e^{-bd^2}\right)v(G)$, where $b$ is a constant and $d$ is the upper bound of the average degree of $G$. This bound was subsequently improved by Conlon, Fox, and Sudakov \cite{CFS2014} to $\mathrm{fvs}(G)\le\left(\frac{1}{2} -(2^7d^2)^{-4d}\right)v(G)$. Meanwhile, Alon, Mubayi, and Thomas \cite{AMT2001} proved that every triangle-free graph $G$ satisfies $\mathrm{fvs}(G)\le \frac{e(G)}{4}$. The current best bound for \Cref{conj:bipartite} is $\mathrm{fvs}(G)\leq\frac{3v(G)-3}{7}$, due to Wang, Xie, and Yu \cite{WXY2017}.

More generally, planar graphs of large girth are also considered. For a connected planar graph $G$ of girth at least $g$, it is easy to see that $\mathrm{fvs}(G)<\frac{2e(G)}{g}$. Dross, Montassier, and Pinlou~\cite{DMP2016} provided a non-trivial bound of $\mathrm{fvs}(G)\le \frac{4e(G)}{3g}$, and proposed the following conjecture.

\begin{conjecture}[Dross, Montassier, and Pinlou~\cite{DMP2016}]\label{conj:dmp}
Every planar graph $G$ of girth at least $g$ satisfies $\mathrm{fvs}(G)\leq \frac{e(G)}{g}$.
\end{conjecture}

A corresponding vertex-based bound of $\mathrm{fvs}(G)\le \frac{v(G)-2}{g-2}$ can be derived from Euler's formula for planar graphs together with the girth condition, that is, $e(G)\leq \frac{g(v(G)-2)}{g-2}$.  
\Cref{conj:dmp} was subsequently confirmed in the case $g=5$ by Kelly and Liu~\cite{KL2017}. More precisely, they proved that every connected planar graph of girth at least $5$ satisfies $\mathrm{fvs}(G)\le \frac{2e(G)-v(G)+2}{7},$
which implies both $\mathrm{fvs}(G)\le \frac{e(G)}{5}$ and $\mathrm{fvs}(G)\le \frac{v(G)-2}{3}$. Tang and Diao~\cite{TD2025} proved that every connected outerplanar graph $G$ of girth $g$ satisfies
$\mathrm{fvs}(G)\leq \frac{e(G)}{g},$
with equality if and only if $G$ is a cycle of length $g$. Ma and Ren~\cite{MR2025} proved that for $g\ge 8$, if a planar graph $G$ has girth at least $2(g-4)^2+6$, then $\mathrm{fvs}(G)\leq \frac{e(G)}{g}$; they also established the following bounds: $\mathrm{fvs}(G)\leq \frac{3v(G)-6}{8}$ for planar graphs of girth at least $6$, and $\mathrm{fvs}(G)\leq \frac{3v(G)-6}{10}$
for planar graphs of girth at least $8$. 
Most recently, Dreyer, Pinlou, and Valicov~\cite{DPV2026} investigated an analogue of the problem for planar digraphs with prescribed digirth, and obtained the upper bound $\mathrm{fvs}(D)\leq \frac{v(D)-2}{g-2}$ for every planar digraph $D$ of digirth $g\geq 3$.

\medskip
In this note, we confirm \Cref{conj:dmp}. The main result is stated as follows.

\begin{theorem}\label{thm:main}
Let $G$ be a planar graph of girth at least $g$ and with $b$ bridges. Then $$\mathrm{fvs}(G)\leq \left\lfloor \frac{e(G)-b}{g}\right\rfloor.$$ 
\end{theorem}

We remark that the bound is tight. Indeed, for every integer $g$, the cycle $C_g$ of length $g$ attains equality in the bound $\frac{e(C_g)}{g}$.

A closely related question in signed graph theory, that of making a signed graph negative-cycle-free, corresponds to the concept of the frustration index. A signed graph $(G, \Sigma)$ is an underlying graph $G$ assigned with a negative edge set $\Sigma\subseteq E(G)$. We call $\Sigma$ a \emph{signature} of $(G, \Sigma)$. The \emph{switching} operation of $(G, \Sigma)$ at an edge cut $X$ is to change the edge set $\Sigma$ into $\Sigma \Delta X$ (that is the symmetric difference of $\Sigma$ and $X$). Two signatures $\Sigma$ and $\Sigma'$ on $G$ are \emph{switching equivalent} if one can be obtained from the other by switching at some edge cut. The \emph{frustration index} of $(G, \Sigma)$, denoted by $F(G, \Sigma)$, is defined to be the minimum number of negative edges in any switching-equivalent signature, that is,
$$F(G, \Sigma) = \min\limits_{\Sigma'} \{|\Sigma'|: \Sigma' ~\text{is switching equivalent to}~\Sigma\},$$ and, equivalently, the frustration index of $(G, \Sigma)$ is the minimum number of edges whose deletion results in a balanced subgraph of $(G,\Sigma)$. Moreover, we denote the \emph{maximum frustration index} of a graph $G$ by $$F_{\max}(G) = \max\limits_\Sigma \{F(G, \Sigma):  \Sigma \text{~is a signature on~}G\}.$$

The minimum size of a feedback vertex set and the (maximum) frustration index measure two different types of obstructions in a graph. The former is the minimum number of vertices whose deletion results in a forest, while the latter is the minimum number of edges whose deletion results in a balanced graph. In general, these two parameters are not comparable; for example, the Wagner graph $H$ in~\Cref{fig:twisted cube} satisfies $\mathrm{fvs}(H)>F_{\max}(H)$, whereas a $2k$-wheel $W_{2k}$ satisfies $\mathrm{fvs}(W_{2k})=2$ and $F_{\max}(W_{2k})=k$, and thus $\mathrm{fvs}(W_{2k})<F_{\max}(W_{2k})$ for $k>2$. 
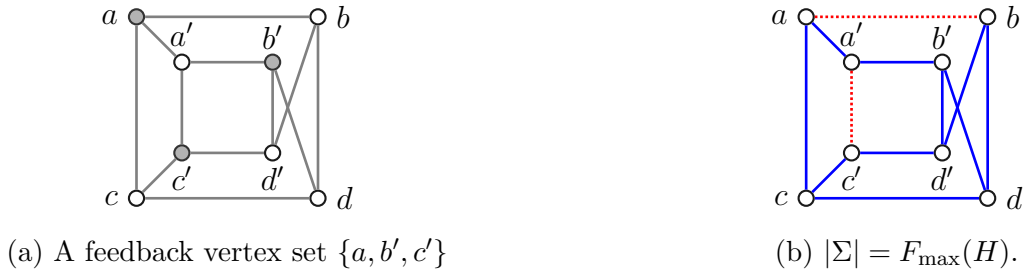
\begin{figure}[htbp]
    \centering

    \begin{subfigure}[t]{0.46\textwidth}
        \centering
        \begin{tikzpicture}[
            >=latex,
            scale=0.8,
            roundnode/.style={
                circle,
                draw=black!90,
                thick,
                minimum size=2mm,
                inner sep=0pt
            },
            fsnode/.style={
                circle,
                draw=black!80,
                fill=gray!60,
                thick,
                minimum size=2mm,
                inner sep=0pt
            }
        ]

            \node[fsnode]    (a)  at (0,0) {};
            \node[roundnode] (b)  at (3,0) {};
            \node[roundnode] (c)  at (0,-3) {};
            \node[roundnode] (d)  at (3,-3) {};

            \node[roundnode] (a1) at (0.75,-0.75) {};
            \node[fsnode]    (b1) at (2.25,-0.75) {};
            \node[fsnode]    (c1) at (0.75,-2.25) {};
            \node[roundnode] (d1) at (2.25,-2.25) {};

            \node[left=1mm]  at (a) {$a$};
            \node[right=1mm] at (b) {$b$};
            \node[left=1mm]  at (c) {$c$};
            \node[right=1mm] at (d) {$d$};

            \node[above=0.7mm] at (a1) {$a'$};
            \node[above=1pt]   at (b1) {$b'$};
            \node[below=1pt]   at (c1) {$c'$};
            \node[below=1pt]   at (d1) {$d'$};

            \draw[line width=1pt, gray]
                (a)--(b)--(d)--(c)--(a)--(a1)--(c1)--(c);

            \draw[line width=1pt, gray]
                (a1)--(b1)--(d1)--(c1);

            \draw[line width=1pt, gray] (b)--(d1);
            \draw[line width=1pt, gray] (d)--(b1);

        \end{tikzpicture}

        \caption{A feedback vertex set $\{a,b',c'\}$}
        \label{fig:G-fvs}
    \end{subfigure}
    \hfill
    \begin{subfigure}[t]{0.46\textwidth}
        \centering
        \begin{tikzpicture}[
            >=latex,
            scale=0.8,
            roundnode/.style={
                circle,
                draw=black!90,
                thick,
                minimum size=2mm,
                inner sep=0pt
            },
            posedge/.style={
                blue,
                line width=1pt
            },
            negedge/.style={
                red,
                densely dotted,
                line width=1pt
            }
        ]

            \node[roundnode] (a)  at (0,0) {};
            \node[roundnode] (b)  at (3,0) {};
            \node[roundnode] (c)  at (0,-3) {};
            \node[roundnode] (d)  at (3,-3) {};

            \node[roundnode] (a1) at (0.75,-0.75) {};
            \node[roundnode] (b1) at (2.25,-0.75) {};
            \node[roundnode] (c1) at (0.75,-2.25) {};
            \node[roundnode] (d1) at (2.25,-2.25) {};

            \node[left=1mm]  at (a) {$a$};
            \node[right=1mm] at (b) {$b$};
            \node[left=1mm]  at (c) {$c$};
            \node[right=1mm] at (d) {$d$};

            \node[above=0.7mm] at (a1) {$a'$};
            \node[above=1pt]   at (b1) {$b'$};
            \node[below=1pt]   at (c1) {$c'$};
            \node[below=1pt]   at (d1) {$d'$};

            \draw[negedge] (a)--(b);
            \draw[negedge] (a1)--(c1);

            \draw[posedge] (b)--(d)--(c)--(a);
            \draw[posedge] (a)--(a1);
            \draw[posedge] (c1)--(c);
            \draw[posedge] (a1)--(b1)--(d1)--(c1);
            \draw[posedge] (b)--(d1);
            \draw[posedge] (d)--(b1);

        \end{tikzpicture}
        \caption{$|\Sigma|=F_{\max}(H)$.}
        \label{fig:G-neg}
    \end{subfigure}

    \caption{The Wagner graph $H$ with $F_{\max}(H)=2$ and $\operatorname{fvs}(H)=3$.}
    \label{fig:twisted cube}
\end{figure}

However, when restricted to planar graphs, we discover a connection between these two parameters and prove the following main result. 

\begin{theorem}\label{thm:fvs}
For any planar graph $G$, $\mathrm{fvs}(G)\le F_{\max}(G)$.
\end{theorem}
This bound is tight. We present an infinite family of $2$-edge-connected signed planar graphs $G_n$ that satisfies that $\mathrm{fvs}(G_n)=F_{\max}(G)$. Each $G_n$ is formed by connecting the components $F_1, F_2, \ldots, F_n$ by adding positive edges $x_iy_{i-1}$ with indices taken modulo $n$. Here $F_i$ is a $2$-subdivision of $K_4$ where the two subdivided edges are adjacent, as depicted in~\Cref{fig:k4sub}.

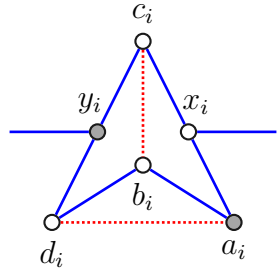
\begin{figure}[htbp]
  \centering
  \begin{tikzpicture}[
    >=latex,
    scale=0.8,
    roundnode/.style={
      circle,
      draw=black!90,
      thick,
      minimum size=2mm,
      inner sep=0pt
    },
    fsnode/.style={
      circle,
      draw=black!90,
      fill=gray!70,
      thick,
      minimum size=2mm,
      inner sep=0pt
    }
  ]
    \node[roundnode] (c) at (0,1.8) {};
    \node[roundnode] (d) at (-1.5,-1.2) {};
    \node[fsnode]    (a) at (1.5,-1.2) {};
    \node[roundnode] (b) at (0,-0.25) {};
    \node[fsnode]    (y) at (-0.75,0.3) {};
    \node[roundnode] (x) at (0.75,0.3) {};

    \draw[fill=white,line width=0.2pt] (0,1.8) node[above=1mm] {$c_i$};
    \draw[fill=white,line width=0.2pt] (-1.5,-1.2) node[below=1mm] {$d_i$};
    \draw[fill=white,line width=0.2pt] (1.5,-1.2) node[below=1mm] {$a_i$};
    \draw[fill=white,line width=0.2pt] (0,-0.25) node[below=1mm] {$b_i$};
    \draw[fill=white,line width=0.2pt] (-0.75,0.3) node[left=1mm][above=1mm] {$y_i$};
    \draw[fill=white,line width=0.2pt] (0.75,0.3) node[right=1mm][above=1mm] {$x_i$};

    \draw[line width=1pt, blue] (d)--(y)--(c)--(x)--(a);
    \draw[line width=1pt, blue] (d)--(b)--(a);
    \draw[line width=1pt, red, densely dotted] (c)--(b);
    \draw[line width=1pt, red, densely dotted] (d)--(a);
    \draw[line width=1pt, blue] (y)--(-2.2,0.3);
    \draw[line width=1pt, blue] (x)--(2.2,0.3);
  \end{tikzpicture}
  \caption{$G_n$ with a minimum feedback vertex set $\{a_1,y_1,\cdots,a_n,y_n\}$, and a maximum signature $\{a_1d_1,b_1c_1,\cdots ,a_nd_n,b_nc_n\}$.}
  \label{fig:k4sub}
\end{figure} 

For the maximum frustration index of a planar graph, the following result can be directly obtained from the dual setting using $T$-joins. Together with \Cref{thm:fvs}, it implies \Cref{thm:main}.

\begin{theorem}\label{prop:fmax}
Let $G$ be a simple planar graph of girth at least $g$ and with $b$ bridges. Then
$$
F_{\max}(G) \leq \lfloor \frac{e(G)-b}{g}\rfloor.
$$
\end{theorem}

We observe that if a planar graph $G$ with girth $g$ has an even number of faces, all of length $g$, then its dual admits a perfect matching $M$ whose corresponding dual edges in $G$ form exactly a maximum signature of $G$. Consequently, such a graph $G$ satisfies $F_{\max}(G)=\frac{f(G)}{2}=\frac{e(G)}{g}$. 


\medskip
We now introduce some terminologies at the end of this section. For any two edge subsets $F_1, F_2\subseteq E(G)$, we use $F_1\Delta F_2$ to denote the symmetric difference of $F_1$ and $F_2$, that is, $F_1\Delta F_2:=(F_1\cup F_2)\setminus (F_1\cap F_2)$. For $X\subseteq V(G)$, let $\delta_G(X)$ denote the set of edges with exactly one end in $X$, where $\delta_G(X)$ is called an \emph{edge cut} of $G$. A \emph{bridge} is an edge cut with exactly one edge. We call an edge subset $F\subseteq E(G)$ \emph{even-degree} if the subgraph induced by $F$ is an even-degree subgraph.

Throughout this paper, we assume that $G$ is a simple connected plane graph with a fixed planar embedding. The \emph{dual graph} of a plane graph $G$, denoted $G^*$, is a plane graph with faces of $G$ as its vertex set, and the edges of $G^*$ correspond to the edges of $G$ as follows: for each edge $e$ of $G$, there is an edge $e^*$ of $G^*$ joining the vertices corresponding to the two faces incident with $e$; if $e$ is incident with only one face (that is, if e is a bridge), then $e^*$ is a loop. So we can see that $v(G^*)=f(G)$ (where $f(G)$ denotes the number of faces of $G$), and $e(G^*)=e(G)$. The dual $(G^*,\Sigma^*)$ of a signed plane graph $(G,\Sigma)$ is defined as follows: the underlying graph is $G^*$, and $e^*$ belongs to $\Sigma^*$ when $e\in \Sigma$.

We have the following observation about $\mathrm{fvs}(G)$ and $F_{\max}(G)$.
\begin{observation}\label{obs:bridge}
For any bridge $e$ of $G$, $\mathrm{fvs}(G)=\mathrm{fvs}(G-e)$, and $F_{\max}(G)=F_{\max}(G-e)$.
\end{observation}

Hence, for bounding these two parameters, we may restrict our attention to $2$-edge-connected graphs. In \Cref{sec:eq-fvs}, we provide some relations between $F_{\max}(G)$ and $T$-joins, and then prove~\Cref{prop:fmax}; and in \Cref{sec:fvs}, we prove~\Cref{thm:fvs}.

\section[Tjoins]{The frustration index and the $T$-joins}\label{sec:eq-fvs}

Let $G$ be a connected plane graph and let $G^*$ be its dual graph. For any edge subset $F\subseteq E(G)$, let $F^*:=\{e^*: e\in F\}\subseteq E(G^*)$. Note that $|F^*|=|F|$. For any vertex subset $X\subseteq V(G)$, since $(\delta_G(X))^*=(\delta_G(v_1))^*\Delta\cdots\Delta (\delta_G(v_k))^*$ for all $v_i\in X$ and each of $(\delta_G(v_k))^*$ induces a cycle in $G^*$, we have that $(\delta_G(X))^*$ is an even-degree edge set. 

For $F\subseteq E(G)$, let $O_{G}(F)=\{v\in V(G): d_F(v)\text{ is odd}\}$. Given a vertex subset $T\subseteq V(G)$ with $|T|$ even, an edge subset $F$ is called a \emph{$T$-join} if $O_{G}(F)=T$. For any connected graph $G$ and any even vertex subset $T\subseteq V(G)$, a $T$-join exists, and thus let $\tau_G(T)$ denote the minimum size of a $T$-join in $G$, that is, $$\tau_G(T):=\min\{\abs{F}: F\subseteq E(G),\ O_{G}(F)=T\}.$$

\begin{observation}\label{obs:dual-signatures}
Let $G$ be a connected plane graph with a fixed embedding and let $G^*$ denote its dual graph. Two edge subsets $F_1$ and $F_2$ of $E(G)$ are switching-equivalent signatures on $G$ if and only if $O_{G^*}(F_1^*)=O_{G^*}(F_2^*)$.
\end{observation}
\begin{proof}
By definition, $F_1$ and $F_2$ are switching-equivalent if and only if $F_1\Delta F_2=\delta_G(X)$ for some vertex subset $X$. Moreover, note that for any $F_1,F_2\subseteq E(G)$, $F_1\Delta F_2$ is an even-degree edge set if and only if $O_{G}(F_1)=O_{G}(F_2)$.

Assume that $F_1$ and $F_2$ are switching-equivalent. By the planarity, $F_1^*\Delta F_2^*=(F_1\Delta F_2)^*=(\delta_G(X))^*$ for some vertex subset $X$. Hence, $F_1^*\Delta F_2^*$ is an even-degree edge set, and thus $O_{G^*}(F_1^*)=O_{G^*}(F_2^*)$. 

Conversely, assume that $O_{G^*}(F_1^*)=O_{G^*}(F_2^*)$. Then $(F_1\Delta F_2)^*=F_1^*\Delta F_2^*$ is an even-degree edge set in $G^*$. By the planar duality, $F_1\Delta F_2$ is an edge cut in $G$, and thus $F_1$ and $F_2$ are switching-equivalent.
\end{proof}

The following result is folklore and can also be found in~\cite{KI1978}. For completeness, we include a proof.

\begin{lemma}[Katai and Iwai~\cite{KI1978}]\label{lem:tjoin-fmax}
For any connected plane graph $G$ and its dual graph $G^*$, $$F_{\max}(G)=\max\limits_T\{\tau_{G^*}(T): T\subseteq V(G^*), \abs{T}\text{~even} \}.$$
\end{lemma}
\begin{proof}
Let $G$ be a connected plane graph and let $G^*$ be its dual graph. For a signature $\Sigma$ on $G$ and a switching-equivalent signature $\Sigma'$ of $\Sigma$, by~\Cref{obs:dual-signatures}, $O_{G^*}(\Sigma^*)=O_{G^*}((\Sigma')^*)$ and thus let $T_\Sigma:=O_{G^*}(\Sigma^*)=O_{G^*}((\Sigma')^*)$. Moreover, for each edge set $F\subseteq E(G^*)$ such that $O_{G^*}(F)=T_\Sigma$, $F^*$ is a signature on $G$ switching-equivalent to $\Sigma$.
So we have that
{\small 
$$
F(G,\Sigma)=\min\limits_{\Sigma'\subseteq E(G)}\{|\Sigma'|:\Sigma'\text{ switching equivalent to }\Sigma\}=\min\limits_{F\subseteq E(G^*)}\{|F|:O_{G^*}(F)=T_\Sigma\}=\tau_{G^*}(T_\Sigma).
$$} Conversely, for every even set $T\subseteq V(G^*)$, by \Cref{obs:dual-signatures}, any two $T$-joins $F_1$ and $F_2$ in $G^*$ imply two switching-equivalent signatures $F_1^*$ and $F_2^*$ on $G$. Let $J^*:=F^*_1$ be one $T$-join. Moreover, for any signature $\Sigma$ on $G$ that is switching-equivalent to $J^*$, $\Sigma^*$ is a $T$-join in $G^*$. Thus 
{\small $$
\tau_{G^*}(T)=\min\limits_{F\subseteq E(G^*)}\{|F|:O_{G^*}(F)=T\}=\min\limits_{\Sigma\subseteq E(G)}\{|\Sigma|:\Sigma\text{~switching-equivalent to $J^*$}\}=F(G,J^*).
$$
}
Therefore, $$F_{\max}(G)=\max\limits_\Sigma \{F(G, \Sigma):  \Sigma \text{~is a signature on~}G\}=\max\limits_T\{\tau_{G^*}(T): T\subseteq V(G^*), \abs{T}\text{~even} \}.$$ We complete the proof.
\end{proof}

For a connected graph $G$ and an even set $T\subseteq V(G)$, the Edmonds--Johnson theorem~\cite{EJ1973} gives the following polyhedral description of a minimum $T$-join. By applying this result through planar duality, we can directly obtain an upper bound on the maximum frustration index of a planar graph of large girth.

\begin{theorem}[Edmonds and Johnson~\cite{EJ1973}]\label{thm:Edmonds--Johnson}
For a connected graph $G$ and an even set $T\subseteq V(G)$,
\begin{equation}\label{eq:tjoin-polytope}
    \tau_G(T)
    =\min\left\{
        \sum_{e\in E(G)}x_e:
        \begin{array}{l}
        x_e\in \mathbb{R}^+,~~\forall e\in E(G);\\[1mm]
        \sum\limits_{e\in \delta_G(U)} x_e\geq 1
        \text{~for every $U\subseteq V(G)$ with~}\abs{U\cap T}\text{ odd.}
        \end{array}
    \right\}.
\end{equation}  
\end{theorem}

Now we are ready to present the proof of~\Cref{prop:fmax}.

\begin{proof}[Proof of~\Cref*{prop:fmax}]
By \Cref{obs:bridge}, it suffices to prove that $F_{\max}(G)\leq \lfloor \frac{e(G)}{g}\rfloor$ for every $2$-edge-connected planar graph $G$ of girth at least $g$. Indeed, bridges can be deleted without affecting the frustration index, and hence we only need to consider the $2$-edge-connected components of it.

Let $G$ be a $2$-edge-connected plane graph and $G^*$ be its dual plane graph. Let $T$ be an arbitrary even vertex subset of $V(G^*)$. We assign $x_e:=\frac{1}{g}$ for every edge $e\in E(G^*)$. By the planar duality, every edge cut $\delta_{G^*}(U)$ corresponds to an even-degree subgraph (induced by $(\delta_{G^*}(U))^*$) in $G$, which must contain a cycle, say $C_U$. Since $G$ is of girth at least $g$, $|\delta_{G^*}(U)|=|(\delta_{G^*}(U))^*|\geq |C_U|\geq g$. This implies that $\sum\limits_{e\in \delta_{G^*}(U)} x_e\geq \frac{1}{g}|C_U|\ge 1$. So such an assignment $x_e$ is feasible for Conditions~(\ref{eq:tjoin-polytope}). Hence, applying~\Cref{thm:Edmonds--Johnson} to $G^*$, $$\tau_{G^*}(T)\le \sum_{e\in E(G^*)}x_e=\frac{e(G^*)}{g}=\frac{e(G)}{g}.$$ Since the above bound applies to any arbitrary even set $T$, \Cref{lem:tjoin-fmax} implies that $F_{\max}(G)=\max\{\tau_{G^*}(T): T\subseteq V(G^*), ~|T|~\text{even} \}\le \frac{e(G)}{g}$ and naturally $F_{\max}(G)\leq \lfloor \frac{e(G)}{g}\rfloor$. 
\end{proof}

\section{The frustration index and the feedback vertex set}\label{sec:fvs}

To establish a comparison between the maximum frustration index and the minimum size of a feedback vertex set, we need the following parameter. Let $$\mu(G):=\max\limits_{F}\{\abs{F}:F\subseteq E(G)\text{ and $\abs{F\cap E(C)}\le \frac{e(C)}{2}$ for every cycle $C$ of $G$}\}.$$
The following folklore fact relates $\mu(G)$ to the size of $T$-joins in $G$. This result can be found, for example, in~\cite{K1962}. For completeness, we provide a short proof.

\begin{lemma}[Kwan~\rm \cite{K1962}]\label{lem:join-fmax}
For any connected graph $H$, 
$$\mu(H)=\max\limits_T\{\tau_H(T): T\subseteq V(H), ~|T|~\text{even} \}.$$
\end{lemma}
\begin{proof}
Let $H$ be a connected graph and let $T\subseteq V(H)$ be an even set. Assume that $J$ is a minimum $T$-join with respect to $|J|$. If there is a cycle $C$ in $H$ such that $\abs{J\cap E(C)}>\frac{e(C)}2$, then $J\Delta E(C)$ is a smaller $T$-join compared with $J$, contradicting the minimality. So every minimum $T$-join satisfies $\abs{J\cap E(C)}\le\frac{e(C)}2$. 
    
Conversely, let $T\subseteq V(H)$ be an even set. Assume that $J$ is a $T$-join such that $\abs{J\cap E(C)}\le\frac{e(C)}2$ for every cycle $C$ of $H$, and $J'$ is another $T$-join. Since $T=O_H(J)=O_H(J')$,  $J\Delta J'$ is an even-degree edge set and, hence, $H[J\Delta J']$ can be decomposed into edge-disjoint cycles $C_1,C_2,\cdots, C_s$. Since every edge in $C_i$ lies in exactly one of $J$ or $J'$, and $\abs{J\cap E(C_i)}\le\frac{e(C_i)}2$ for each $i$, we have that $\abs{J\cap E(C_i)}\leq \abs{J'\cap E(C_i)}$. Therefore, $\abs{J\setminus J'}\le \abs{J'\setminus J}$, and $J$ is a minimum $T$-join.

Therefore, for any $T\subseteq V(H)$ with $|T|$ even, $J$ is a minimum $T$-join if and only if $\abs{J\cap E(C)}\le\frac{e(C)}2$ for every cycle $C$ of $H$. We complete the proof. 
\end{proof}

The following corollary is a direct consequence of \Cref{lem:tjoin-fmax} and \Cref{lem:join-fmax}.

\begin{corollary}\label{col:fmaxeq}
For any connected plane graph $G$ and its dual graph $G^*$, $F_{\max}(G)=\mu(G^*).$
\end{corollary}

We now introduce the notion of ear decompositions following the formulation of Frank~\cite{F1993}. An \emph{ear decomposition} of a connected graph $G$, denoted by $\mathcal{E}_G$, is a sequence of graphs $G_0\subset G_1\subset \cdots \subset G_s$ such that $G_0$ consists of a single vertex and no edges, $G_s=G$, and for each $i\in\{1,2,\ldots,s\}$, the graph $G_i$ is obtained from $G_{i-1}$ by adding a path $P_i$ whose endpoints are vertices (which might be the same) in $G_{i-1}$ and whose internal vertices (if any) are disjoint from $V(G_{i-1})$. Each such path $P_i$ is called an \emph{ear} of the decomposition. Note that an ear is allowed to have length $1$. An ear is \emph{even} or \emph{odd} according to whether its length (i.e., its number of edges) is even or odd. A classical result of Robbins~\cite{R1939} states that a connected graph is $2$-edge-connected if and only if it admits an ear decomposition. For any $2$-edge-connected graph $G$, let $\varphi(G)$ denote the minimum number of even ears among all ear decompositions $\mathcal{E}_G$ of $G$. Frank~\cite{F1993} gave the following key relation between the two parameters $\mu(G)$ and $\varphi(G)$.

\begin{theorem}[Frank~\cite{F1993}]\label{thm:frank}
For any $2$-edge-connected graph $G$, $\varphi(G)=2\mu(G)-v(G)+1.$
\end{theorem}

Next we consider the duals of the plane graphs appearing in an ear decomposition of a $2$-edge-connected plane graph $G$.
\emph{Splitting a vertex $v$} is an operation that replaces $v$ with two new vertices $v_1$ and $v_2$, and assigns each edge incident with $v$ to exactly one of $v_1$ and $v_2$. For any face $F$ of a plane graph $G$, we use $v_F$ to denote the corresponding vertex in $G^*$.

\begin{observation}\label{obs:addear}
Let $G$ be a $2$-edge-connected plane graph and let $F_0$ be a face of $G$. Let $G'$ be formed from $G$ by adding a path $P$ of length $p$ with the two endpoints on the vertex set of $F_0$. Then the dual graph $(G')^*$ of $G'$ is obtained from $G^*$ by splitting $v_{F_0}$ into $v^1_{F_0}$ and $v^2_{F_0}$, and adding $p$ parallel edges connecting $v^1_{F_0}$ and $v^2_{F_0}$.
\end{observation}

The next lemma is the final ingredient in the proof of \Cref{thm:fvs}. The original proof is attributed to Kir\'aly and Kisfaludi-Bak~\cite{KK2012} using the notion of dual-critical graphs. Here, we provide a direct inductive proof based on ear decompositions.

\begin{lemma}\label{lem:key}
Let $G$ be a $2$-edge-connected plane graph and let $G^*$ be its dual graph.  Let $\mathcal{E}_{G^*}$ be an ear decomposition of $G^*$ with $k$ even ears. There exists an edge subset $D\subseteq E(G)$ with $|D|=k$ such that $G-D$ has a spanning tree $T$ satisfying that $E(G-D)\setminus E(T)$ can be partitioned into pairs of incident edges.
\end{lemma}

\begin{proof}
Let $G$ be a $2$-edge-connected plane graph and let $G^*$ be its dual graph. It is straightforward that $G^*$ is also a $2$-edge-connected plane graph. By Robbins' result~\cite{R1939}, $G^*$ admits an ear decomposition $G^*_0\subset G^*_1\subset\cdots \subset G^*_s$ containing $k$ even ears. For each $i\in \{1,\ldots, s\}$, assume that $G^*_i$ is obtained from $G^*_{i-1}$ by adding an ear $P_i$ of length $p_i$. We may assume that the planar embedding of $G_i^*$ is chosen to extend that of $G^*_{i-1}$. For each $i\in \{0,1,\ldots, s\}$, let $G_i$ denote the dual graph of $G^*_i$, and $k_i$ denote the number of even ears in the partial ear decomposition $G^*_0\subset \cdots \subset G^*_i$ of $G^*_i$. Clearly, $k_0\leq k_1\leq \cdots \leq k_s=k$.

We prove by induction on $i$ that there exists a set $D_i\subseteq E(G_i)$ with $|D_i|=k_i$ such that $G_i-D_i$ has a spanning tree $T_i$ satisfying that
$E(G_i-D_i)\setminus E(T_i)$ can be partitioned into pairs of incident edges.

The graph $G^*_0$ consists of one vertex and no edges, and so does $G_0$. Thus, the assertion holds for the case $i=0$ with $D_0=\emptyset$. Suppose that the statement holds for $i-1$. Since $G^*_i$ is obtained from $G^*_{i-1}$ by adding an ear $P_i$, by~\Cref{obs:addear}, $G_i$ is obtained from $G_{i-1}$ by splitting a vertex $z$ into two vertices $z_1$ and $z_2$, followed by adding $p_i$ parallel edges between $z_1$ and $z_2$. We identify $z_1$ with $z$, and thus regard $z_2$ as the new vertex in $G_i$. 



We first define $D_i$ inductively. Set $D_0=\emptyset$. For each $i\ge 1$, if $p_i$ is odd, then $k_i=k_{i-1}$ and we set
$D_i:=D_{i-1}$. If $p_i$ is even, then $k_i=k_{i-1}+1$, and we choose an arbitrary edge $e_i$ connecting $z_1$ and $z_2$, and set $D_i:=D_{i-1}\cup\{e_i\}$. In both cases, $|D_i|=k_i$.

Note that $G_i-D_i$ is obtained from $G_{i-1}-D_{i-1}$ by splitting $z$ into $z_1,z_2$ and adding an odd number of parallel edges between them. Indeed, if $p_i$ is even, then there are $p_i-1$ parallel edges between $z_1$ and $z_2$ in $G_i-D_i$; if $p_i$ is odd, then there are $p_i$ parallel edges between $z_1$ and $z_2$ in $G_i-D_i$.  
In order to construct the corresponding $T_i$ for each $i$, we prove the following claim.

\begin{claim}\label{claim:order}
For each $i\in \{0,1,\ldots, s\}$, the graph $G_i-D_i$ admits a vertex ordering $\mathcal{O}_i=(v_1, v_2, \ldots, v_{n_i})$ such that for every $\ell\geq 2$, the number of edges joining $v_\ell$ to $\{v_1,\ldots,v_{\ell-1}\}$ is odd.  
\end{claim}
\begin{proof*}[Proof of \Cref{claim:order}]
We prove the statement by induction on $i$. The case $i=0$ is trivial. Assume that the statement holds for $i-1$, and let $\mathcal{O}_{i-1}$ be a vertex ordering of $G_{i-1}-D_{i-1}$ satisfying the required condition. In particular, the vertex $z$ to be split, either is the first vertex in the ordering, or has an odd number of incident edges to its earlier neighbors in $\mathcal{O}_{i-1}$.

We construct $\mathcal{O}_i$ from $\mathcal{O}_{i-1}$ as follows. If $z$ is the first vertex in $\mathcal{O}_{i-1}$, we replace $z$ by $z_1,z_2$. If $z_1$ is incident with an odd number of edges joining it to the earlier neighbors of $z$, then we replace $z$ in $\mathcal{O}_{i-1}$ by $z_1,z_2$. Otherwise, we replace $z$ by $z_2,z_1$.
We claim that $\mathcal{O}_i$ is a required vertex ordering of $G_i-D_i$. Indeed, recall that $z_1$ and $z_2$ are connected by an odd number of parallel edges in $G_i-D_i$. Exactly one of $z_1$ and $z_2$ is incident with an odd number of edges joining it to the earlier neighbors of $z$, and the other vertex is incident with an even number of such edges. We complete the proof of the claim.
\end{proof*}

For each $i\in \{0,1,\ldots, s\}$, based on $\mathcal{O}_i$, we construct a spanning tree $T_i$ of $G_i-D_i$ as follows. For every $\ell\geq 2$, choose one edge $e_\ell$ joining $v_\ell$ to one of its earlier neighbors in $\mathcal{O}_i$. Since every $v_\ell$ has at least one earlier neighbor, the chosen edges form a connected graph with $n_i-1$ edges, and hence constitute a spanning tree of $G_i-D_i$.

Moreover, when restricted to $G_i-(D_i\cup E(T_i))$, every vertex $v_\ell$ with $\ell\ge2$ has an even number of edges joining it to earlier vertices in the ordering $\mathcal{O}_i$. Therefore, these edges can be paired arbitrarily at each vertex $v_\ell$, yielding a partition of $E(G_i-D_i)\setminus E(T_i)$ into pairs of incident edges.
\end{proof}

By \Cref{obs:bridge}, we may assume that $G$ is $2$-edge-connected, and thus its dual graph $G^*$ is also $2$-edge-connected. Now we can prove \Cref{thm:fvs}.

\begin{proof}[Proof of~\Cref{thm:fvs}]
  Since $G$ is $2$-edge-connected, $G^*$ is a $2$-edge-connected plane graph. Let $\mathcal{E}_{G^*}$ be an ear decomposition of $G^*$ with exactly $\varphi(G^*)$ even ears. By~\Cref{lem:key}, there exists $D\subseteq E(G)$ with $\abs{D}=\varphi(G^*)$ such that $G-D$ admits a spanning tree $T$ satisfying that $E(G-D)\setminus E(T):=R$ can be partitioned into pairs of incident edges.
  
  We form a feedback vertex set $S$ as follows. For every edge $e\in D$, choose one vertex of $e$; for every incident edge pair $\{e_1,e_2\}$ of $R$, choose a common endpoint of $e_1$ and $e_2$. Let $S$ be the set of all chosen vertices. Since every edge in $E(G)\setminus E(T)$ has an endpoint in $S$, $G-S\subseteq T$. Consequently, $G-S$ is a forest and $S$ is a feedback vertex set. It is easy to see that $|S|\le\frac{|R|}{2}+|D|$. 
  
  Since $T$ is a spanning tree of $G-D$, $|R|=e(G)-\abs{D}-(v(G)-1)=e(G)-v(G)+1-\varphi(G^*).$ Therefore, $$\mathrm{fvs}(G)\le |S|\le\frac{|R|}{2}+|D|=\frac{e(G)-v(G)+1+\varphi(G^*)}{2}=\frac{v(G^*)-1+\varphi(G^*)}{2}=\mu(G^*),$$ where the second last equality follows from the planarity of $G$ and Euler's formula (namely, $e(G)-v(G)+1=v(G^*)-1$), and the last equality follows from~\Cref{thm:frank}. By \Cref{col:fmaxeq}, $\mu(G^*)=F_{\max}(G)$, thus $\mathrm{fvs}(G)\le F_{\max}(G)$, which completes the proof.
\end{proof}

\section*{Declaration of AI usage}
During the development and preparation of this manuscript, the authors used ChatGPT on a limited basis to explore possible approaches to selected parts of the mathematical arguments (namely, locating the relevant arguments and the proof of \Cref{lem:key} in~\cite{KK2012}, and identifying the reference~\cite{F1993}), and to improve the language and presentation of the paper. All AI-assisted arguments were carefully checked, revised, and independently verified by the authors. Except for these limited uses, all proofs in the paper and their presentation were carried out by the authors, who take full responsibility for the correctness and final content of the manuscript.


{\footnotesize
\bibliographystyle{amsplain}

}

\end{document}